\documentclass[10pt]{amsart}
\usepackage{amssymb,amsthm,amsmath,amsfonts}
\usepackage{hyperref}
\usepackage{pstricks}
\usepackage{graphicx}
\usepackage{pst-node}

\def\R{{\mathbb R}}
\def\Rr{{\mathcal R}}
\def\C{{\mathbb C}}

\def\Z{{\mathbb Z}}

\def\<{\langle}
\def\>{\rangle}

\def\U{{\mathbb U}}
\def\V{{\mathbb V}}

\def\I{\mathbb I}

\def\X{\mathbb X}

\def\Y{\mathbb Y}

\newcommand{\be}{\begin{equation}}
\newcommand{\ee}{\end{equation}}
      \newtheorem{theorem}{Theorem}[section]
       \newtheorem{proposition}[theorem]{Proposition}
       
       \newtheorem{lemma}[theorem]{Lemma}
       \newtheorem{remark}{Remark}[section]
\newtheorem{definition}{Definition}[section]

\title{On the Lukacs property for free random variables}
\author[K. Szpojankowski]{Kamil Szpojankowski}
\address[K. Szpojankowski]{Wydzia\l{} Matematyki i Nauk Informacyjnych\\
Politechnika Warszawska\\
ul. Koszykowa 75\\
00-662 Warszawa, Poland}
\email{k.szpojankowski@mini.pw.edu.pl}


\begin{document}
\begin{abstract}
The Lukacs property of the free Poisson distribution is studied here. We prove that if free $\X$ and $\Y$ are free Poisson distributed with suitable parameters, then $\X+\Y$ and $\left(\X+\Y\right)^{-\frac{1}{2}}\X\left(\X+\Y\right)^{-\frac{1}{2}}$ are free. As as an auxiliary result we give joint cumulants of $\X$ and $\X^{-1}$ for free Poisson distributed $\X$. We also study the Lukacs property of the free gamma distribution.
\end{abstract}
\maketitle
\section{Introduction}
The celebrated Lukacs theorem \cite{Lukacs} in the classical probability theory gives a characterization of independent random variables $X$ and $Y$ which have the Gamma distribution $G(a,p)$ and $G(a,q)$ respectively, by independence of $V=X+Y$ and $U=\frac{X}{X+Y}$. By the Gamma distribution $G(a,p)$ we mean here a probability distribution given by a probability density function
\begin{align*}
\frac{a^{p}}{\Gamma(p)}x^{p-1}e^{-ax}\mathit{I}_{(0,\infty)}(x),\,\mbox{   } a,p>0
\end{align*}

This result was generalized in many directions. An important direction of such investigations was relaxing the assumptions of independence of $U$ and $V$. The same characterization holds true when instead of assuming independence, constancy of the first and second conditional moment of $U$ given by $V$ is assumed, see \cite{BolHark} or \cite{LahaLukacs} for more general so called Laha-Lukacs regressions. One can also consider other powers of $U$ (see eg. \cite{WesolGamma}). \\
It is easy to see that $U$ defined as above has the Beta distribution of the first kind $\beta_{I}(p,q)$, where by the Beta distribution we mean probability distribution with the density
\begin{align*}
\frac{x^{p-1}(1-x)^{q-1}}{\beta(p,q)}\mathit{I}_{(0,1)}(x).
\end{align*}
Another way of generalizing Lukacs theorem are so called dual Lukacs regressions introduced in \cite{BobWes2002Dual}, where authors proved that if $U$ and $V$ are independent and such that $i$th and $j$th conditional moment of $Y=V(1-U)$ given by $X=UV$ are constant for $(i,j)\in\{(-1,-2),(-1,1),(1,2)\}$ then $U$ has a Beta distribution and $V$ has a Gamma distribution.

Lukacs property was also studied in context of random matrices, where it turns out that this property characterizes the Wishart distribution (see \cite{CasalisLetac96,BobWes2002,BoutHassMass,OlkinRubin62,OlkinRubin64, LetacMassam, LetacWes2011}).

In the present paper we study analogues of the Lukacs property in free probability. Free probability and notion of free independence of non-commutative random variables was introduced by Voiculescu in \cite{VoiculescuAdd}. It turns out during the development of the theory that classical and free probability are deeply related. One of the links between theories is so called asymptotic freeness, which roughly speaking says that large independent random matrices under the state being expectation of normalized trace are close to free random variables (see \cite{VoiculescuLimit,SpeicherNC}). Another fact which gives relation between this two theories are Berkovici Pata bijections (see \cite{BerkoviciPata}) which give bijection between infinitely divisible measures under classical and free convolution. Also an important link is given in \cite{SpeicherNC,NicaSpeicherLect} where it is proved that free cumulants can be define using lattice of non-crossing partitions while classical cumulants are define using lattice of all partitions.

It was also noticed in the literature that characterizations of independent and free random variables are closely related. However this relationship is not completely understood.  There are many examples of classical characterizations which have free counterparts. A basic example is Bernstein theorem \cite{Bernstein} which characterizes the normal distribution by independence of $X+Y$ and $X-Y$ for independent $X$ and $Y$. It's free analogue was proved in \cite{NicaChar} and says that if $\X$ and $\Y$ are free then $\X+\Y$ and $\X-\Y$ are free if and only if $\X$ and $\Y$ have a semicircular distribution.

Lukacs theorem was also studied in context of  free probability. It turns out that in context of characterizations the role of the Gamma distribution in the free probability plays a Marchenko-Pastur distribution also known as a free Poisson distribution. In \cite{BoBr2006} authors proved a free analogue of Laha-Lukacs regressions, they described family of free Meixner distributions by assumptions on two first conditional moments of $\X$ given by $\X+\Y$, one of the cases of Laha-Lukacs regressions is free analogue of Lukacs regressions. Laha-Lukacs regressions and related characterizations were also studied in \cite{EjsmontLL,EjsmontMP,EjsmontTD}. Dual Lukacs regressions in the free probability were studied in our previous works \cite{SzpojanWesol,SzpDLRNeg}. \\
Lukacs property in free probability is well studied, but there is a significant gap in the study of the free Lukacs property, namely there is no proof that for free $\X$ and $\Y$ free Poisson distributed, random variables $\U=(\X+\Y)^{-\frac{1}{2}}\X(\X+\Y)^{-\frac{1}{2}}$ and $\V=\X+\Y$ are free. In Section 3 we give a proof of this fact.

We have to note that the closest result in this direction was proved in \cite{CaCa} where authors proved asymptotic freeness of $\bf U=(X+Y)^{-\frac{1}{2}}X(X+Y)^{-\frac{1}{2}}$ and $\bf V=X+Y$ for $\bf X, Y$ independent, complex Wishart distributed. Using results from \cite{CaCa} in \cite{SzpojanWesol} a result in the opposite direction is proved: for $\U$ free Binomial distributed and $\V$ free Poisson distributed random variables $\X=\V^{\frac{1}{2}}\U\V^{\frac{1}{2}}$ and $\Y=\V^{\frac{1}{2}}(\I-\U)\V^{\frac{1}{2}}$ are free.

Here we do not use asymptotic freeness technique for the proof of free Lukacs property of the free Poisson distribution. Our proof is based on direct calculation of joint cumulants of $\U$ and $\V$. The computation involves joint cumulants of $\X$ and $\X^{-1}$ for the free Poisson distributed $\X$, for which we find closed formula.
The form of joint cumulants of $\X$ and $\X^{-1}$ leads also to a new characterization of the free Poisson distribution.

We also prove that if $\X$ is free Poisson distributed then $\X^{-1}$ belongs to the free Gamma family defined in \cite{BoBr2006}. This observation together with knowledge of joint cumulants of $\X$ and $\X^{-1}$ allows us to give an answer to a question raised in \cite{BoBr2006}, if for free, positive, identically free gamma distributed $\X$ and $\Y$, random variables $\V=\X+\Y$ and $\U=(\X+\Y)^{-1}\X^2(\X+\Y)^{-1}$ are free. The answer turns out to be negative.

The paper is organized as follows: in Section 2 we give basics of the free probability, Section 3 is devoted to find joint cumulants of $\X$ and $\X^{-1}$ for free Poisson distributed $\X$. We also prove in this section a characterization of the free Poisson distribution closely related to the formula for joint cumulants of $\X$ and $\X^{-1}$. In Section 4 we study Lukacs property of free Poisson and free Gamma distributions.
\section{Preliminary}
In this section we will give basic definitions and facts necessary for understanding the results, for more comprehensive introduction to the free probability consult \cite{NicaSpeicherLect} or \cite{VoiDykNica}.\\

A non-commutative $*$-probability space is a pair $(\mathcal{A},\varphi)$, where $\mathcal{A}$ is a unital algebra over $\C\,$ and $\varphi:\mathcal{A}\to\C$ is a linear functional satisfying $\phi(\I)=1$ and $\varphi\left(\X^*\X\right)\geq 0$, if $\mathcal{A}$ is a $C^*$-algebra then $(\mathcal{A},\varphi)$ is called $C^*$-probability space. Any element $\X$ of $\mathcal{A}$ is called a (non-commutative) random variable.

The $*$-distribution $\mu$ of a self-adjoint element $\X\in\mathcal{A}\subset\mathcal{B}(H)$ is a probabilistic measure on $\R$ such that $$\varphi(\X^r)=\int_{\R}\,t^r\,\mu(dt)\qquad \forall\,r=1,2,\ldots$$

Unital subalgebras $\mathcal{A}_i\subset \mathcal{A}$, $i=1,\ldots,n$, are said to be freely independent if $\varphi(\X_1,\ldots,\X_k)=0$ for $\X_j\in \mathcal{A}_{i(j)}$, where $i(j)\in\{1,\ldots,n\}$, such that $\varphi(\X_j)=0$, $j=1,\ldots,k$, if neighbouring elements are from different subalgebras, that is $i(1)\ne i(2)\ne \ldots \ne i(k-1)\ne i(k)$. Similarly, random variables $\X,\,\Y\in\mathcal{A}$ are free (freely independent) when subalgebras generated by $(\X,\,\I)$ and $(\Y,\,\I)$ are freely independent (here $\I$ denotes identity operator).

Let $\chi=\{B_1,B_2,\ldots\}$ be a  partition of the set of numbers $\{1,\ldots,k\}$. A partition $\chi$ is a crossing partition if there exist distinct blocks $B_r,\,B_s\in\chi$ and numbers $i_1,i_2\in B_r$, $j_1,j_2\in B_s$ such that $i_1<j_1<i_2<j_2$. Otherwise $\chi$ is called a non-crossing partition. The set of all non-crossing partitions of $\{1,\ldots,k\}$ is denoted by $NC(k)$.

For any $k=1,2,\ldots$, (joint) cumulants of order $k$ of non-commutative random variables $\X_1,\ldots,\X_n$ are defined recursively as $k$-linear maps $\mathcal{R}_k:\mathcal{A}^k \to\C\,$ through equations
\begin{align}
\label{cum_def}
\varphi(\X_1\ldots\X_m)=\sum_{\chi\in NC(m)}\,\prod_{B\in\chi}\,\mathcal{R}_{|B|}(\X_i,\,i\in B)
\end{align}
holding for any $m=1,2,\ldots$,
with $|B|$ denoting the size of the block $B$.

Freeness can be characterized in terms of behaviour of cumulants in the following way: Consider unital subalgebras $(\mathcal{A}_i)_{i\in I}$ of an algebra $\mathcal{A}$ in a non-commutative probability space $(\mathcal{A},\,\varphi)$. Subalgebras $(\mathcal{A}_i)_{i\in I}$ are freely independent iff for any $n=2,3,\ldots$ and for any $\X_j\in\mathcal{A}_{i(j)}$ with $i(j)\in I$, $j=1,\ldots,n$ any $n$-cumulant
$$
\mathcal{R}_n(\X_1,\ldots,\X_n)=0
$$
if there exists a pair $k,l\in\{1,\ldots,n\}$ such that $i(k)\ne i(l)$.

In particular we will need a result saying that joint cumulants of $\I$ which is the unit of the algebra $\mathcal{A}$ with any other non-commutative random variable vanish (see \cite{NicaSpeicherLect}, Proposition 11.15).
\begin{proposition}
\label{zero_ident}
Let $(\mathcal{A},\varphi)$ be non-comutative probability space, and let $\X_1,\ldots,\X_n\in\mathcal{A},$ $n\geq 2$. Then we have
\begin{align*}
\mathcal{R}_n\left(\X_1,\ldots,\X_n\right)=0
\end{align*}
whenever there exists $k\in\{1,\ldots,n\}$ s.t. $\X_k=\I$.
\end{proposition}

A non-commutative random variable $\X$ is said to be free-Poisson distributed if it has Marchenko-Pastur(or free-Poisson) distribution $\nu=\nu(\lambda,\alpha)$ defined by the formula
\be\label{MPdist}
\nu=\max\{0,\,1-\lambda\}\,\delta_0+\min\{1,\lambda\} \tilde{\nu},
\ee
where $\lambda\ge 0$ and the measure $\tilde{\nu}$, supported on the interval $(\alpha(1-\sqrt{\lambda})^2,\,\alpha(1+\sqrt{\lambda})^2)$, $\alpha>0$ has the density (with respect to the Lebesgue measure)
$$
\tilde{\nu}(dx)=\frac{1}{2\pi\alpha x}\,\sqrt{4\lambda\alpha^2-(x-\alpha(1+\lambda))^2}\,dx.
$$
The parameters $\lambda$ and $\alpha$ are called the rate and the jump size, respectively.

Assume that $\X$ has free Poisson distribution with $\lambda>1$, note that in this case $\X$ is invertible since it's spectrum do not contain 0. One can easily find that the probability density function of the distribution of $\X^{-1}$ is given by
\begin{align}
\label{free_pois_inv}
f_{\X^{-1}}(x)=\frac{(\lambda-1)\sqrt{\frac{4\lambda}{\alpha^2(\lambda-1)^4}-\left(x-\frac{\lambda+1}{\alpha(\lambda-1)^2}\right)^2}}
{2\pi x^2}\mathit{I}_{\left(\frac{1}{\alpha(1+\sqrt{\lambda})^2},\frac{1}{\alpha(1-\sqrt{\lambda})^2}\right)}(x).
\end{align}
A random variable $\Y$ defined by $\Y=\alpha(\lambda-1)^{\frac{3}{2}}\X^{-1}-(\lambda-1)^{\frac{1}{2}}\I$ has the standardized free gamma law $\mu_{2a,a^2}$ defined in \cite{BoBr2006}, where $a=\frac{1}{\sqrt{\lambda-1}}$. To see that one can compare the probability density function of $\Y$ and the probability density function of the standardized free gamma distribution (page 62 in \cite{BoBr2006}). Cumulants of the standardized free gamma law are (see Remark 5.7 in \cite{BoBr2006})  $\mathcal{R}_1(\mu_{2a,a^2})=0$, $\mathcal{R}_k(\mu_{2a,a^2})=C_{k-1}a^{k-2}$ for $k\geq 2$, where $C_k,\,k\geq$ 1 are Catalan numbers. By multilinearity of cumulants we obtain that cumulants of $\X^{-1}$ are equal to
\begin{align}
\label{free_gamma_cumulants}
\mathcal{R}_k(\X^{-1})=C_{k-1}\frac{1}{\alpha^k(\lambda-1)^{2k-1}},\,  k\geq 1.
\end{align}
Catalan numbers are defined by
\begin{align}
\label{Cat1}
C_k=\frac{1}{k+1}{2k\choose k},\, k\geq 0
\end{align}
They can be also equivalently defined by the recurrence $C_0=1$
\begin{align}
\label{Cat2}
C_k=\sum_{i=1}^kC_{i-1}C_{k-i}.
\end{align}

By traciality of $\varphi$ one can easily proved the following lemma.
\begin{lemma}
\label{tr_cum}
Let $(\mathcal{A},\varphi)$ be *-probability space. Assume that $\varphi$ is tracial, then for any $n\in\mathbb{N}$
\begin{align*}
\mathcal{R}_n\left(\X_1,\ldots,\X_n\right)=\mathcal{R}_n\left(\X_n,\X_1,\ldots,\X_{n-1}\right).
\end{align*}
\end{lemma}
%\begin{proof}
%We will prove this fact by induction. For $n=1$ it is obvious.\\
%Assume that \eqref{tr_cum} is true for $k<=n-1$.\\
%We will compute moment $\varphi\left(a_1\cdot\ldots\cdot a_n\right)$. We can rewrite this moment as a sum of cumulants over non-crossing partitions.
%\begin{align*}
%\varphi\left(a_1\cdot\ldots\cdot a_n\right)=\sum_{\pi \in NC(n)} \mathcal{R}_{\pi}(a_1,\ldots_,a_n)=&\sum_{\pi \in NC(n)} \mathcal{R}_{P(\pi)}(a_n,\ldots_,a_1)\\&-R_n\left(a_n,a_1,\ldots,a_{n-1}\right)+R_n\left(a_1,\ldots,a_n\right)
%\end{align*}
%Where last equation holds by the induction hypothesis. Actually in every partition we change order of variables only in block containing $n$. For all cumulants corresponding to the partition different from $\mathit{1}_n$, by induction hypothesis this cumulant corresponding to changed partition is equal to cumulant corresponding to unchanged partition. Than we subtract appropriate
%cumulants corresponding to the partition $\mathit{1}_n$.\\
%By the fact that $P$ is bijection we have that
%\begin{align*}
%\sum_{\pi \in NC(n)} \mathcal{R}_{P(\pi)}(a_n,\ldots_,a_1)=\sum_{\pi \in NC(n)} \mathcal{R}_{\pi}(a_n,\ldots_,a_1)=\varphi\left(a_n a_1\ldots,a_{n-1}\right).
%\end{align*}
%So we get $\varphi\left(a_1\cdot\ldots\cdot a_n\right)=\varphi\left(a_na_1\ldots a_{n-1}\right)-R_n\left(a_n,a_1,\ldots,a_{n-1}\right)+R_n\left(a_1,\ldots,a_n\right)$, by tracial property of $\varphi$ we obtain, that $R_n\left(a_n,a_1,\ldots,a_{n-1}\right)=R_n\left(a_1,\ldots,a_n\right)$
%\end{proof}
In the next lemma we will denote by $\sigma_n$ partition of the set $\{1,2,\ldots,n\}$, which consists of $n-1$ blocks of the form $\{(1,2),(3),\ldots,(n)\}$, so first and second element are in the same block, and all other elements are singletons and by $\mathit{1}_n$ partition $\{(1,2,\ldots,n)\}$. This lemma is Theorem 11.12 form \cite{NicaSpeicherLect} rewritten for partition $\sigma_n$.
\begin{lemma}
\label{kumulant_product}
Let $(\mathcal{A},\varphi)$ be *-probability space, and $\X_1,\X_2,\ldots,\X_n\in\mathcal{A}$.\\
Then
\begin{align}
\mathcal{R}_n(\X_1\cdot \X_2,\X_3,\ldots, \X_{n+1})=\sum_{\pi\in NC(n+1):\, \pi \vee \sigma_{n+1} = \mathit{1}_{n+1}}\mathcal{R}_{\pi}(\X_1,\X_2,\ldots,\X_{n+1}).
\end{align}
It can be rewritten as
\begin{align}
\mathcal{R}_n(\mathbb{X}_1\cdot \X_2,\X_3,\ldots,\mathbb{X}_{n+1})=&\sum_{i=1}^{n-1}\mathcal{R}_i\left(\mathbb{X}_2,\mathbb{X}_3,\ldots,\X_{i+1}\right)
\Rr_{n+1-i}
\left(\X_1,\X_{i+2},\ldots,\mathbb{X}_{n+1}\right)\\
\nonumber+&\Rr_n(\X_2,\ldots,\X_{n+1})\Rr_1(\X_1)\\
\nonumber+&\mathcal{R}_{n+1}\left(\mathbb{X}_1,\mathbb{X}_2,\ldots,\mathbb{X}_{n+1}\right).
\end{align}
\end{lemma}
\section{Joint cumulants of $\X$ and $\X^{-1}$}
In this section we will find joint cumulants of $\X$ and $\X^{-1}$ for an invertible, free Poisson distributed non-commutative random variable $\X$. This result will be the main tool in proof of the Lukacs property of the free Poisson distribution in Section 4. Additionally we will also prove characterization of the free Poisson distribution related to the joint cumulants of $\X$ and $\X^{-1}$
\begin{proposition}
\label{kumulanty}
Let $\X$ have free Poisson distribution with $\lambda>1$, and $\alpha=1$. Assume that $m\geq 1$ then
\begin{align}
\mathcal{R}_{i_1+\ldots+i_m+m}(\X^{-1},\underbrace{\X,\ldots,\X}_{i_1},\X^{-1},
\underbrace{\X,\ldots,\X}_{i_2},\X^{-1},\ldots,\X^{-1},\underbrace{\X,\ldots,\X}_{i_m})=\\
\nonumber
\begin{cases} 0 \,& if \,  \exists k\in \{1,\ldots,m\},\, i_k>1 \\
(-1)^{i_1+\ldots+i_m}\mathcal{R}_{m}(\X^{-1}) & if\, \forall k\in\{1,\ldots,m\},\, i_k\leq 1.
\end{cases}
\end{align}
\end{proposition}
\begin{proof}
Let $n$ be the length of cumulant i.e. in the above cumulant $n=i_1+\ldots+i_m+m$. We will prove the Proposition by induction over $n$. For cumulant of the length $1$ it is clear that Proposition holds true. Since the inductive step will work for cumulants of length greater than or equal to $3$, we have to check if the proposition holds true for cumulants of length 2. For $\Rr_2\left(\X^{-1},\X^{-1}\right)$ lemma is obviously true. It is enough to check that $\Rr_2\left(\X^{-1},\X\right)=-\Rr_1\left(\X^{-1}\right)$. By Lemma \ref{kumulant_product} we obtain
\begin{align*}
1=\Rr_1(\I)=\Rr_1\left(\X^{-1}\X\right)=
\Rr_2\left(\X^{-1},\X\right)+\Rr_1\left(\X^{-1}\right)\Rr_1\left(\X\right)=
\Rr_2\left(\X^{-1},\X\right)+\frac{\lambda}{\lambda-1}.
\end{align*}
From the above equation we see that $\Rr_2\left(\X^{-1},\X\right)=-\frac{1}{\lambda-1}=-\Rr_1\left(\X^{-1}\right)$.

Assume that proposition holds true for $l\leq n-1$.
Note that by the Lemma \ref{tr_cum} joint cumulant of variables $\X$ and $\X^{-1}$ can be transformed to the form, where first two variables are respectively $\X^{-1},\X$, it is impossible only in the case when only $\X^{-1}$ appears in the cumulant. In that case proposition is obviously true.

To prove the proposition we will consider two cases.\\
\textbf{Case 1}.\\
We will prove that a cumulant in which there is at least one pair of neighbouring $\X$'s is equal to zero. We may assume that last (reading from left) pair of neighbouring $\X$'s is at positions $k, k+1$, where $k=2,3,\ldots,n-1$. We will prove that $$\mathcal{R}_n\left(\X^{-1},\X,\ldots,\underbrace{\X,\X}_{k,k+1},\ldots\right)=0$$.

From Proposition \ref{zero_ident} we get
\begin{align*}
0=\mathcal{R}_{n-1}(\I,\ldots,\X,\X,\ldots)=\mathcal{R}_{n-1}(\X^{-1}\X,\ldots,\X,\X,\ldots)=\mathcal{R}_{n-1}(\Y_1\cdot \Y_2,\Y_3,\ldots,\Y_{n}).
\end{align*}
Now we expand right hand side of the above equation using the Lemma \ref{kumulant_product}
\begin{align*}
0=\sum_{i=1}^{n-2}&\mathcal{R}_i\left(\Y_2,\Y_3,\ldots,\Y_{i+1}\right)\mathcal{R}_{n-i}
\left(\Y_1,\Y_{i+2},\ldots,\Y_n\right)\\+&\Rr_{n-1}(\Y_2,\ldots,\Y_n)\Rr_1(\Y_1)+\mathcal{R}_{n}(\Y_1,\ldots,\Y_n).\end{align*}
Note that for $i\leq k-2$ cumulant $\Rr_{n-i}
\left(\Y_1,\Y_{i+2},\ldots,\Y_n\right)$ contains $\Y_1=\X^{-1}$, $\Y_k=\X$ $\Y_{k+1}=\X$, so by the inductive assumption is equal to zero.\\
It is left to prove that
\begin{align}
\label{kum_zero}
\mathcal{R}_{n}(\Y_1, \Y_2,\Y_3,\ldots,\Y_{n})=&-\sum_{i=k-1}^{n-2}\mathcal{R}_i\left(\Y_2,\Y_3,\ldots,\Y_{i+1}\right)
\Rr_{n-i}\left(\Y_1,\Y_{i+2},\ldots,\Y_n\right)\\&-\Rr_{n-1}(\Y_2,\ldots,\Y_n)\Rr_1(\Y_1)
\nonumber\\=&0\nonumber.
\end{align}
We will consider two cases:
\begin{itemize}
\item \textbf{a} --- there is  $j\in\{3,\ldots,k-1\}$ such that $\Y_j=\X^{-1}$,
\item \textbf{b} --- for all $j\in\{3,\ldots,k-1\}$ $\Y_j=\X$.
\end{itemize}

\textbf{Case a}. In the case when there is $j\in\{3,\ldots,k-1\}$ such that$\Y_j=\X^{-1}$, the cumulant $\mathcal{R}_i\left(\Y_2,\Y_3,\ldots,\Y_{i+1}\right)$ for $i=k-1$ contains $\Y_2=\X,\, \Y_j=\X^{-1}$ and $\Y_k=\X$. So by the Lemma \ref{tr_cum} and the inductive assumption this cumulant is equal to zero. For $i\in \{k,k+1,\ldots,n-1\}$ the cumulant $\mathcal{R}_i\left(\Y_2,\Y_3,\ldots,\Y_{i+1}\right)$ contains $\Y_j=\X^{-1},\, \Y_k=\X,\, \Y_{k+1}=\X$, so by the inductive assumption are equal to zero, which completes the proof in the case \textbf{a}.

\textbf{Case b}. Assume that for all $j\in\{3,\ldots,k-1\}$ we have $\Y_j=\X$, when $n> k+1$ (which means that the last pair of neghbouring $\X$ is not at positions $n-1,n$) then we have $\Y_{k+2}=\X^{-1}$, otherwise the chosen pair of $\X$ would not be the last pair of neighbouring $\X$. From this we have that for $i\geq k+1$ the cumulant $\mathcal{R}_i\left(\Y_2,\Y_3,\ldots,\Y_{i+1}\right)$ contains $\Y_k=\X,\,\Y_{k+1}=\X$ and $\Y_{k+2}=\X^{-1}$, so by the inductive assumption and the Lemma \ref{tr_cum} the cumulant $\mathcal{R}_i\left(\Y_2,\Y_3,\ldots,\Y_{i+1}\right)$ is equal to zero. \\

It is left to prove that in the case $n> k+1$ we have
\begin{align*}
\mathcal{R}_{n}(\Y_1,
\Y_2,\Y_3,\ldots,\Y_{n})=-\sum_{i=k-1}^{k}\mathcal{R}_i\left(\Y_2,\Y_3,\ldots,\Y_{i+1}\right)
\Rr_{n-i}
\left(\Y_1,\Y_{i+2},\ldots,\Y_n\right)=0,
\end{align*}
and in the case $n=k+1$ that
\begin{align*}
\mathcal{R}_{n}(\Y_1,
\Y_2,\Y_3,\ldots,\Y_{n})=&-\mathcal{R}_{n-2}\left(\Y_2,\Y_3,\ldots,\Y_{n-1}\right)
\Rr_{2}
\left(\Y_1,\Y_n\right)\\&-\Rr_{n-1}(\Y_2,\ldots,\Y_n)\Rr_1(\Y_1)=0.
\end{align*}
Both cases can be proved similarly. For all $j\in\{2,\ldots,k+1\}$ we have $\Y_j=\X$ and by the assumption $\X$ has free Poisson distribution with parameters $(\lambda,1)$, so cumulants of $\X$ are constant, which means $\mathcal{R}_{k-1}\left(\Y_2,\Y_3,\ldots,\Y_{k}\right)=
\mathcal{R}_{k}\left(\Y_2,\Y_3,\ldots,\Y_{k+1}\right)$. Moreover note that $\Y_1=\X^{-1}$, $\Y_{k+1}=\X$, and on positions $k+1,\ldots,n$ there are no neighbouring $\X$, then by the inductive assumption
$\Rr_{n-k+1}
\left(\Y_1,\Y_{k+1},\ldots,\Y_n\right)=-\Rr_{n-k}
\left(\Y_1,\Y_{k+2},\ldots,\Y_n\right)$ (in the case $n=k+1$ the right hand side equals $-\mathcal{R}_1\left(\X^{-1}\right)$). So the above sum is equal to zero, which completes the proof in the Case 1.

\textbf{Case 2}.\\
We will prove the lemma in case when there are no neighbouring $\X$'s in the cumulant, of course there might be neighbouring $\X^{-1}$. Without loss of generality we may assume that except for the position 2, $\X$ appears in the cumulant exactly $k\leq\lfloor\frac{n}{2}\rfloor-1$ times, on positions $i_1,i_2,\ldots,i_k$, where $4\leq i_1<\ldots<i_k\leq n$ and $\forall j\in\{1,2,\ldots,k-1\}$ we have $i_{j+1}-i_{j}>1$.\\
Taking into account the formula \eqref{free_gamma_cumulants} which gives cumulants of $\X^{-1}$ , we have to prove that $$\mathcal{R}_{n}(\X^{-1},\X,\X^{-1},\ldots)=\frac{(-1)^{k+1}}{(\lambda-1)^{2(n-k-1)-1}}C_{n-k-2}.$$
For $i\in\{1,\ldots,n\}$ by $\Y_i$ we denote the variable on the $i$th position. We proceed similarly as in the previous case,
\begin{align*}
0=&\mathcal{R}_{n-1}(\I,\X^{-1}\ldots,)=\mathcal{R}_{n-1}(\X^{-1}\X,\X^{-1}\ldots,)=\mathcal{R}_{n-1}(\Y_1\Y_2,\ldots,\Y_n)=
\\&\mathcal{R}_n(\Y_1,\Y_2,\ldots,\Y_n)
+\mathcal{R}_1(\Y_2)\mathcal{R}_{n-1}(\Y_1,\Y_3,\ldots,\Y_n)
+\\&\sum_{i=2}^{n-2}\mathcal{R}_i(\Y_2,\Y_3,\ldots,\Y_{i+1})\mathcal{R}_{n-i}(\Y_1,\Y_{i+2},\Y_{i+3},\ldots,\Y_n)
+\Rr_{n-1}(\Y_2,\ldots,\Y_n)\Rr_1(\Y_1).
\end{align*}
Note that $\Y_2=\X$ and $\Y_j=\X$ for $j\in\{i_1,\ldots,i_k\}$. From the fact that in the initial cumulant there were no neighbouring $\X$ and by the inductive assumption, we obtain that terms of the above sum for $i+1\in\{i_1,\ldots,i_k\}$ are equal to $0$, similarly $\Rr_{n-1}(\Y_2,\ldots,\Y_n)\Rr_1(\Y_1)$ is equal zero when $\Y_n=\X$. All other terms are non-zero.\\
Note that $k+1$ variables from $\Y_1,\Y_2,\ldots,\Y_n$ are equal $\X$ and $n-k-1$ are equal $\X^{-1}$.\\
From the above remarks and the inductive assumption we get
\begin{align*}
0=&\mathcal{R}_n(\Y_1,\Y_2,\ldots,\Y_n)+(-1)^k\mathcal{R}_1(\X)\mathcal{R}_{n-k-1}\left(\X^{-1}\right)+
\sum_{i=1}^{n-k-2}(-1)^{k+1}\mathcal{R}_{i}\left(\X^{-1}\right)\mathcal{R}_{n-k-1-i}\left(\X^{-1}\right)\\
=&\mathcal{R}_n(\Y_1,\Y_2,\ldots,\Y_n)+\frac{(-1)^k\lambda}{(\lambda-1)^{2(n-k-1)-1}}C_{n-k-2}+
\sum_{i=1}^{n-k-2}(-1)^{k+1}\frac{C_{i-1}C_{n-k-i-2}}{(\lambda-1)^{2i-1+2(n-k-1-i)-1}}\\
=&\mathcal{R}_n(\Y_1,\Y_2,\ldots,\Y_n)+\frac{(-1)^k\lambda}{(\lambda-1)^{2(n-k-1)-1}}C_{n-k-2}+
\frac{(-1)^{k+1}}{(\lambda-1)^{2(n-k-1)-2}}\sum_{i=1}^{n-k-2}C_{i-1}C_{n-k-2-i}\\
=&\mathcal{R}_n(\Y_1,\Y_2,\ldots,\Y_n)+\frac{(-1)^k\lambda}{(\lambda-1)^{2(n-k-1)-1}}C_{n-k-2}+
\frac{(-1)^{k+1}}{(\lambda-1)^{2(n-k-1)-2}}C_{n-k-2}\\
=&\mathcal{R}_n(\Y_1,\Y_2,\ldots,\Y_n)+\frac{(-1)^{k}}{(\lambda-1)^{2(n-k-1)-1}}C_{n-k-2}.
\end{align*}
In the equation one before last we used recurrence for Catalan numbers \eqref{Cat2}. \\
From the above equation we see that $\mathcal{R}_n(\Y_1,\Y_2,\ldots,\Y_n)=\frac{(-1)^{k+1}}{(\lambda-1)^{2(n-k-1)-1}}C_{n-k-2}$, which completes the proof of the lemma.
\end{proof}
The next lemma gives characterization of the invertible, free Poisson distributed random variable in language of joint cumulants of $\X$ and $\X^{-1}$.
\begin{remark}
\label{uwazka}
If $\X$ is invertible and such that for $n\geq 1$ we have
\begin{align}
\label{uw_kum}
\Rr_{n+1}\left(\X,\X^{-1},\X^{-1},\ldots,\X^{-1}\right)=-\Rr_{n}\left(\X^{-1}\right),
\end{align}
then $\Rr_n\left(\X^{-1}\right)=\frac{1}{\left(\Rr_1(\X)-1\right)^{2n-1}}C_{n-1}$  for $n\geq 1$.\\
In particular $\X^{-1}$ has the free Gamma distribution and $\X$ has the free Poisson distribution.
\end{remark}
\begin{proof}
We will prove the remark inductively. Using the asumption and Lemma \ref{kumulant_product} we obtain
\begin{align*}
1=\Rr_1(\I)=\Rr_1\left(\X\X^{-1}\right)=\Rr_1(\X)\Rr\left(\X^{-1}\right)+
\Rr_2\left(\X,\X^{-1}\right)=\Rr\left(\X^{-1}\right)\left(\Rr_1(\X)-1\right),
\end{align*}
hence $\Rr_1\left(\X^{-1}\right)=\frac{1}{\Rr_1(\X)-1}$.\\
Assume that the remark holds true for $k\leq n-1$. Using respectively: the Lemma \ref{kumulant_product}, the assumption, the inductive assumption and the recurrence for the Catalan numbers \eqref{Cat2} we obtain
\begin{align*}
0&=\Rr_n\left(\X\X^{-1},\X^{-1},\ldots,\X^{-1}\right)=\\
&\Rr_{n+1}\left(\X,\X^{-1},\ldots,\X^{-1}\right)+\Rr_1(\X)\Rr_n\left(\X^{-1}\right)+
\sum_{i=1}^{n-1}\Rr_i\left(\X^{-1}\right)\Rr_{n+1-i}\left(\X,\X^{-1},\ldots,\X^{-1}\right)=\\
&\left(\Rr_1(\X)-1\right)\Rr_n\left(\X^{-1}\right)
-\sum_{i=1}^{n-1}\Rr_i\left(\X^{-1}\right)\Rr_{n-i}\left(\X^{-1}\right)=
\\ &\left(\Rr_1(\X)-1\right)\Rr_n\left(\X^{-1}\right)-\frac{1}{(\Rr_1(\X)-1)^{2n-2}}
\sum_{i=1}^{n-1}C_{i-1}C_{n-1-i}=\\
&\left(\Rr_1(\X)-1\right)\Rr_n\left(\X^{-1}\right)-\frac{1}{(\Rr_1(\X)-1)^{2n-2}}C_{n-1}.
\end{align*}
Which means that $\Rr_n\left(\X^{-1}\right)=\frac{1}{\left(\Rr_1(\X)-1\right)^{2n-1}}C_{n-1}$.
\end{proof}
\section{The Lukacs property in free probability}
Before we start the proof of the Lukacs theorem, we need to introduce so called Kreweras complement on non-crossing partition and useful formula for computing cumulants of products of free random variables. The following definition and theorem can be found in \cite{NicaSpeicherLect} (def. 9.21, th. 14.4), see also
\cite{Kreweras}.
\begin{definition}
Let $\pi$ be a partition of set $\{1,2,...,n\}$. Consider set $\{1,\overline{1},2,\overline{2},\ldots,n,\overline{n}\}$.
The Kreweras complement of the partition $\pi$ denoted by $K(\pi)$ is the biggest partition $\sigma\in NC(\overline{1},\overline{2},\ldots,\overline{n})$ such that $\pi\cup\sigma\in NC(1,\overline{1},2,\overline{2},\ldots,n,\overline{n})$.
\end{definition}
For example if $\pi=\{(1,2),(3,4)\}$, then $K(\pi)=\{(\overline{1}),(\overline{2},\overline{4})(\overline{3})\}.$
Partitions $\pi$ and $K(\pi)$ are illustrated below.
\\
 \setlength{\unitlength}{.20cm}
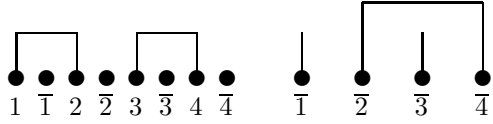
 \begin{equation*}
 \begin{picture}(-10,8)

\put(-15,3){\circle*{1}} \put(-15,3){\line(0,1){3}}

\put(-13,3){\circle*{1}}

\put(-11,3){\circle*{1}} \put(-11,3){\line(0,1){3}}

\put(-9,3){\circle*{1}}

\put(-7,3){\circle*{1}}\put(-7,3){\line(0,1){3}}

\put(-5,3){\circle*{1}}

\put(-3,3){\circle*{1}}\put(-3,3){\line(0,1){3}}

\put(-1,3){\circle*{1}}

\put(-15,6){\line(1,0){4}}

\put(-7,6){\line(1,0){4}}

\put(-15.5,0.5){1}
\put(-13.5,0.5){$\overline{1}$}
\put(-11.5,0.5){2}
\put(-9.5,0.5){$\overline{2}$}
\put(-7.5,0.5){3}
\put(-5.5,0.5){$\overline{3}$}
\put(-3.5,0.5){4}
\put(-1.5,0.5){$\overline{4}$}

\put(4,3){\circle*{1}} \put(4,3){\line(0,1){3}}

\put(8,3){\circle*{1}} \put(8,3){\line(0,1){5}}

\put(12,3){\circle*{1}}\put(12,3){\line(0,1){3}}

\put(16,3){\circle*{1}}\put(16,3){\line(0,1){5}}

\put(8,8){\line(1,0){8}}

\put(3.5,0.5){$\overline{1}$}

\put(7.5,0.5){$\overline{2}$}

\put(11.5,0.5){$\overline{3}$}

\put(15.5,0.5){$\overline{4}$}

 \end{picture}
 \end{equation*}

\begin{theorem}
\label{cum_prod}
Let $\{\X_1,\ldots,\X_n\}$ and $\{\Y_1,\ldots,\Y_n\}$ be free, then
\begin{align}
\mathcal{R}_n(\X_1\Y_1,\X_2\Y_2,\ldots,\X_n\Y_n)=\sum_{\pi\in NC(n)}R_{\pi}\left(\X_1,\X_2,\ldots,\X_n\right)R_{K(\pi)}\left(\Y_1,\Y_2,\ldots,\Y_n\right).
\end{align}
\end{theorem}
Now we are ready to prove the following theorem, which is the main result of the paper.
\begin{theorem}
\label{lukacs}
Let $(\mathcal{A},\varphi)$ be $C^*$-probability space and $\varphi$ be faithful, tracial state. Let $\X,\Y\in\mathcal{A}$ be free, both free Poisson distributed with parameters $(\lambda,\alpha)$ and $(\kappa,\alpha)$ respectively, where $\lambda+\kappa>1$, then random variables $\U=(\X+\Y)^{-1/2}\X(\X+\Y)^{-1/2}$ and $\V=\X+\Y$ are free.
\end{theorem}
\begin{proof}
First we note that it is enough to prove that mixed cumulants of $\U,\V$ vanish. Since $\alpha$ is a multiplicative constant and cumulants are multilinear, we can assume that $\alpha=1$. First we will prove the theorem with an additional assumption that $\X$ is invertible, which means that $\lambda>1$. In this case we can equivalently prove freeness of $\U^{-1}=(\X+\Y)^{1/2}\X^{-1}(\X+\Y)^{1/2}$ and $\U$. Since $\varphi$ is tracial then any moment $\varphi(P(\U^{-1},\V))$, where $P(x,y)$ is a non-commutative polynomial in variables $x$ and $y$, can be rewritten as $\varphi(P(\mathbb{W}\,,\V))$, where $\mathbb{W}=\X^{-1}(\X+\Y)$. From this and the definition of free cumulants we have that vanishing of joint cumulants of $\U^{-1}=(\X+\Y)^{1/2}\X^{-1}(\X+\Y)^{1/2}$ and $\V=\X+\Y$ is equivalent to vanishing of mixed cumulants of $\X^{-1}(\X+\Y)=\I+\X^{-1}\Y$ and $\X+\Y$. Because of the fact, that any joint cumulant containing $\I$ and any other random variable equals $0$, to prove the theorem (in the case $\lambda>1$) it is enough to prove that all mixed cumulants of random variables $\widetilde{\U}=\X^{-1}\Y$ and $\V=\X+\Y$ vanish.\\
Fix $n\geq 2$ and some mixed cumulant of random variables $\widetilde{\U},\V$ of length $n$. By the Lemma \ref{tr_cum} we can change the order of variables in the cumulant, in order to have $\widetilde{\U}$ at the first position. Without loss of generality we may assume that in this cumulant $\V$ appears $k<n$ times, on positions (after change of the order) $1\leq j_1<j_2<\ldots<j_k\leq n$.
\begin{align}
\label{tw_expand}
\mathcal{R}_n\left(\widetilde{\U},\ldots,\underbrace{\V}_{j_1},\ldots,\underbrace{\V}_{j_k},\ldots\right)=\sum_{Z_{j_1},\ldots,
\Z_{j_k}\in\{\X,\Y\}}\mathcal{R}_n\left(\widetilde{\U},\ldots,\Z_{j_1},\ldots,\Z_{j_k},\ldots\right)
\end{align}
In the first step of the proof we will find these terms of the above sum which are equal to 0.\\
Note that if we  write variables $\X$ and $\Y$ as $\,\X\I,\,\I\Y$ respectively, then we can apply theorem \ref{cum_prod} to the above cumulants.\\
Consider now cumulant of the form $\mathcal{R}_n(\ldots,\X\I,\I\Y,\ldots)$,
expanding it according to the Theorem \ref{cum_prod} we see that either $\X$ and $\I$ are in the same block of partition $\pi$, or $\I$ and $\Y$ are in the same block of partition $K(\pi)$. In both cases the cumulant is equal to 0 by the Proposition \ref{zero_ident}.\\
Fix now $l\geq 2$ and consider kumulant of the length $n>l+1$ which contains $\X^{-1}\Y$ and sequence of random variables consisting of $l$ neighbouring $\X$ and maybe other sequences of $\X\I,\I\Y,\X^{-1}\Y$. From the previous consideration to be non-zero this cumulant must be of the form $\mathcal{R}_n(\ldots,\underbrace{\X\I}_{1\,\overline{1}},
\ldots,\underbrace{\X\I}_{l\,\overline{l}},\underbrace{\X^{-1}\Y}_{l+1\,\overline{l+1}},\ldots).$
Again we will expand this cumulant according to the Theorem \ref{cum_prod}. Since a joint cumulant of $\I$ and any other random variable equals 0, then if in the  partition $K(\pi)$ numbers $\overline{1},\ldots,\overline{l}$ are not singletons, then the above cumulant is equal to zero. Assume that in $K(\pi)$ $\overline{1},\ldots,\overline{l}$ are singletons, then in partition $\pi$ positions $1,2,\ldots,l+1$ are in the same block, in this case one of the cumulant related to $\pi$ contains at least two neighbouring $\X$'s and $\X^{-1}$ so by the Proposition \ref{kumulanty} this cumulant is equal to zero.

From the above remarks we conclude that the only non-zero cumulants in the sum
\eqref{tw_expand} are these, which consists only of $\Y$ and $\X^{-1}\Y$ or if they contain $\X$ at the position $j\in\{2,\ldots,n-1\}$ (after using the Lemma \ref{tr_cum}) then $\X^{-1}\Y$ is at the position $j+1$ is and at the position $j-1$ can not be $\X$.

Recall that in equation \eqref{tw_expand} $\X^{-1}\Y$ appears $n-k$ times. Some of $\X^{-1}\Y$ may have as a left neighbour another $\X^{-1}\Y$, we may assume that $m\leq n-k$ of them do not have as a left neighbour $\X^{-1}\Y$. Above remarks imply that if $\Z_j$ does not have as a right neighbour $\X^{-1}\Y$ then $\Z_j=\Y$, in other case the cumulant is equal to zero. This means that $\X$ can appear only on $m$ positions which are left neighbours of $\X^{-1}\Y$. So we can rewrite the right hand site of \eqref{tw_expand} as
\begin{align}
\label{expand1}
\sum_{Z_{i_1},\ldots,
\Z_{i_m}\in\{\X\I,\I\Y\}}\mathcal{R}_n\left(\X^{-1}\Y,\I\Y,\ldots,\I\Y,\Z_{i_1},
\X^{-1}\Y,\I\Y\ldots,\I\Y,\Z_{i_j},\X^{-1}\Y,\ldots\I\Y,\Z_{i_m}\right).
\end{align}
Fix one term of the above sum, without loss of generality we can assume that $j\in\{0,1,\ldots,m\}$ of $\{\Z_{i_1},\ldots,\Z_{i_m}\}$ are equal $\X\I$. We will expand such term according to the Theorem \ref{cum_prod}
\begin{align}
\label{expand2}
\mathcal{R}_n\left(\X^{-1}\Y,\ldots,\Z_{i_1},\X^{-1}\Y,\ldots,\Z_{i_j},\X^{-1}\Y,\ldots,\Z_{i_m}\right)=
\sum_{\pi\in NC(n)}\Rr_\pi(\X^{-1},\ldots)\Rr_{K(\pi)}(\Y,\ldots).
\end{align}

Note that if $\Z_{i_j}=\X\I$ then $\overline{i_j}$ is singleton in $K(\pi)$ or this cumulant vanishes by the Lemma \ref{zero_ident}. From this we see that $i_j$ and $i_{j+1}$ are in the same block of partition $\pi$, which means that the number related to the position of $\X$ is in the same block with the number related to the position of $\X^{-1}$. Moreover by previous steps in considered cumulant there are no neighbouring $\X$.

Similarly if in the fixed term of the sum \eqref{expand2} $\Z_{i_j}=\I\Y$, then in $\pi$ element $i_j$ is a singleton or this cumulant is equal to 0.

The above remarks leads to a conclusion that partitions $\pi$ for which right hand side of \eqref{expand2} is not equal to zero, can be identified with partitions on $n-k$ element related to the positions of $\X^{-1}$. Other elements are singletons or are in the same block with one of elements related to the position of $\X^{-1}$. A mapping defined below formalizes this observation.

Recall that $\X^{-1}\Y$ appears $n-k$ times in cumulant on the left hand side of \eqref{expand2}, so in $\Rr_{\pi}\left(\X^{-1},\ldots\right)$ variable $\X^{-1}$ appears also $n-k$ times. Let us denote the positions of $\X^{-1}\Y$ by $l_1,\ldots,l_{n-k}$.\\
For any partition $\pi\in NC(n)$ we define $\widetilde{\pi}\in NC(n-k)$ by restriction of $\pi$ to the numbers $l_1,\ldots,l_{n-k}$. This means that $s,t\in\{1,\ldots,n-k\}$ are in the same block of $\widetilde{\pi}$ if and only of $l_s$ and $l_t$ are in the same block of $\pi$.\\
Let $A=\{\pi\in NC(n): \Rr_{\pi}\left(\X^{-1},\ldots\right)\neq 0\,\mbox{and}\,\Rr_{K(\pi)}\left(\Y,\ldots\right)\neq 0\}$. From the previous remarks it follows that the mapping $\pi\to\widetilde{\pi}$ is bijection between $A$ and $NC(n-k)$.
We will prove that
\begin{align}
\label{mapA}
\sum_{\pi\in A}\Rr_\pi(\X^{-1},\ldots)\Rr_{K(\pi)}(\Y,\ldots)=
\sum_{\widetilde{\pi}\in NC(n-k)}(-1)^j\Rr_{\widetilde{\pi}}(\underbrace{\X^{-1},\ldots\X^{-1}}_{n-k})\Rr_{K(\widetilde{\pi})}(\underbrace{\Y,\ldots,\Y}_{n-k}).
\end{align}
If $\pi\in A$ and in $\Rr_\pi\left(\X^{-1},\ldots\right)$ variable $\X$ appears $j$ times then from the Proposition \ref{kumulanty} we get that $\Rr_\pi(\X^{-1},\ldots)=(-1)^j\Rr_{\widetilde{\pi}}(\underbrace{\X^{-1},\ldots\X^{-1}}_{n-k})$.

Recall that in the sum \eqref{expand1} some of $\Z_j$'s was fixed to be $\I\Y$. From this we obtain that after applying the Lemma \ref{tr_cum} every maximal sequence of neighbouring $\I\Y$ has as a left neighbour $\X^{-1}\Y$. In particular if in the left hand side of \eqref{expand1} there is sequence of $l$ neighbouring $\I\Y$ then it must be of the form $\Rr_n\left(\ldots,\X^{-1}\Y,\underbrace{\I\Y,\ldots,\I\Y}_l,\ldots\right)$. It is clear that for $\pi\in A$ positions of $\I$ from $\I\Y$ are singletons, from this it follows that in $K(\pi)$ positions of $\Y$ from $\I\Y$ are in the same block with some $\Y$ coming from $\X^{-1}\Y$, this implies that number of block in $K(\pi)$ is equal to number of block of $K(\widetilde{\pi})$. Since $\Y$ has the free Poisson distribution with $\alpha=1$, then cumulants of $\Y$ are constant. This means that $\Rr_{K(\pi)}(\Y,\ldots)$ depends only on number of blocks in $K(\pi)$. It proves that $\Rr_{K(\pi)}(\Y,\ldots)=\Rr_{K\left(\widetilde{\pi}\right)}(\underbrace{\Y,\ldots,\Y}_{n-k}).$\\
This proves the equation \eqref{mapA}.

Let us return to the sum \eqref{tw_expand}, we can rewrite it summing over possible number of $\X\I$ which gives us
\begin{align*}
\sum_{j=0}^{m}\sum_{\widetilde{\pi}\in NC(n-k)}{m\choose j} (-1)^{j}\mathcal{R}_{\widetilde{\pi}}(\X^{-1},\ldots,\X^{-1})R_{K(\widetilde{\pi})}(\Y,\ldots,\Y),
\end{align*}
after changing the order summation we obtain
\begin{align*}
\sum_{\widetilde{\pi}\in NC(n-k)} \mathcal{R}_{\widetilde{\pi}}(\X^{-1},\ldots,\X^{-1})R_{K(\widetilde{\pi})}(\Y,\ldots,\Y)
\sum_{j=0}^{m}{m\choose j} (-1)^{j}=0.
\end{align*}
Which completes the proof in case $\lambda>1$ and $\kappa>0$.\\
Assume now that $\lambda+\kappa>1$ and $\alpha>0$.\\
Note that the equation \eqref{cum_def} defines cumulants recursively by moments. Any moment of the random variables $\V$ and $\U$ by traciality of $\varphi$ can be expressed as a moment of $\X+\Y$, $\X$, $(\X+\Y)^{-1}$. Since $\X+\Y$ has free Poisson distribution with parameters $(\lambda+\kappa,\alpha)$, and $\lambda+\kappa>1$, then the support of the distribution of $\X+\Y$ is $[\alpha(1-\sqrt{\lambda+\kappa})^2,\alpha(1+\sqrt{\lambda+\kappa})^2]$. As it was pointed at the beginning of the proof it is enough to prove freeness of $\U$ and $\V$ for some fixed $\alpha$ and freeness for other values of $\alpha$ follows from multilinearity of free cumulants. We fix $\alpha>0$ such that the support of the distribution of the random variable $\I-(\X+\Y)$ is contained in $(-1,1)$. Since the support of a random variable is equal to the spectrum and the spectral norm is equal to the norm we have $||\I-(\X+\Y)||<1$ and $(\X+\Y)^{-1}=\sum_{n=0}^\infty \left(\I-(\X+\Y)\right)^n$. This means that any joint moment of $\X+\Y,$ $\X,$ $(\X+\Y)^{-1}$ can be expressed as a series of moments of $\X+\Y$ and $\X$. Moreover by freeness of $\X$ and $\Y$ any joint moment of $\X+\Y$ and $\X$ is a polynomial of moments of $\X$ and moments of $\Y$. Taking into account that moments of free Poisson distribution with parameters $(\lambda,\alpha)$ are the same polynomials in $\lambda,\alpha$ for $\lambda\leq1$ and for $\lambda> 1$, we conclude that cumulants of $\U$ and $\V$ are power series of $\lambda,\alpha,\kappa$, so they are the same analytic function of $\lambda,\kappa$ for $\lambda\leq 1$ and for $\lambda> 1$. This completes the proof.
\end{proof}
Recall that Proposition $\ref{kumulanty}$ gives joint cumulants of $\X$ and $\X^{-1}$ where $\X$ has free Poisson distribution. As it was shown for a positive $\X$ with the free Poisson distribution, $\X^{-1}$ has a free Gamma distribution with the density \eqref{free_pois_inv}. Thus we can also use Proposition $\ref{kumulanty}$ (with swapped roles of $\X$ and $\X^{-1}$) to prove the following result.
\begin{proposition}
\label{free_gamma}
Let $\X$ and $\Y$ be free, identically distributed. Assume that the distribution of $\X$ is free Gamma then $\X+\Y$ and $(\X+\Y)^{-1}\X^2(\X+\Y)^{-1}$ are not free.
\end{proposition}
\begin{remark}
The above proposition gives a negative answer to the question stated in \cite{BoBr2006} after Prop. 3.6. Moreover the above result together with the Proposition 3.6 from \cite{BoBr2006}  imply that there is no such a pair of free, identically distributed random variables, that $\X+\Y$ and $(\X+\Y)^{-1}\X^{2}(\X+\Y)^{-1}$ are free.
\end{remark}
\begin{proof}[Proof of Proposition \ref{free_gamma}]
Similarly as in the proof of the Theorem \ref{lukacs} we can equivalently prove that $\X+\Y$ i $(\X+\Y)\X^{-2}(\X+\Y)$ are not free. Let us note that $(\X+\Y)\X^{-2}(\X+\Y)=(\I+\Y\X^{-1})(\I+\X^{-1}\Y)$. We will prove that
$\mathcal{R}_{2}\left((\I+\Y\X^{-1})(\I+\X^{-1}\Y),\X+\Y\right)\neq 0$.\\
From the Lemma \ref{kumulant_product} we can write
\begin{align*}
\mathcal{R}_{2}\left((\I+\Y\X^{-1})(\I+\X^{-1}\Y),\X+\Y\right)=&\mathcal{R}_{3} \left(\I+\Y\X^{-1},\I+\X^{-1}\Y,\X+\Y\right)\\+&
\mathcal{R}_{2}(\I+\Y\X^{-1},\X+\Y)\mathcal{R}_{1}(\I+\X^{-1}\Y)\\
+&\mathcal{R}_{2}(\I+\X^{-1}\Y,\X+\Y)\mathcal{R}_{1}(\I+\Y\X^{-1})
\end{align*}
Lemma \ref{zero_ident} implies
\begin{align*}
\mathcal{R}_{3} \left(\I+\Y\X^{-1},\I+\X^{-1}\Y,\X+\Y\right)=
\mathcal{R}_{3}\left(\Y\X^{-1},\X^{-1}\Y,\X\right)+
\mathcal{R}_{3}\left(\Y\X^{-1},\X^{-1}\Y,\Y\right).
\end{align*}
Using Lemma \ref{kumulant_product} we get
\begin{align*}
\mathcal{R}_{3}\left(\Y\X^{-1},\X^{-1}\Y,\X\right)&=
\Rr_4\left(\Y,\X^{-1},\X^{-1}\Y,\X\right)+
\Rr_1\left(\X^{-1}\right)\Rr_3\left(\Y,\X^{-1}\Y,\X\right)\\
&+\Rr_2\left(\Y,\X\right)\Rr_2\left(\X^{-1},\X^{-1}\Y\right)+
\Rr_1\left(\Y\right)\Rr_3\left(\X^{-1},\X^{-1}\Y,\X\right).
\end{align*}
We can use Lemma \ref{kumulant_product} once again and by freeness of $\X$ and $\Y$ we obtain 
\begin{align*}
\mathcal{R}_{3}\left(\Y\X^{-1},\X^{-1}\Y,\X\right)=
\Rr_1^2\left(\Y\right)\Rr_3\left(\X^{-1},\X^{-1},\X\right).
\end{align*}
By Lemma \ref{kumulanty} (with swapped roles of $\X$ and $\X^{-1}$) we get 
\begin{align*}
\mathcal{R}_{3}\left(\Y\X^{-1},\X^{-1}\Y,\X\right)=0.
\end{align*}
Similarly
\begin{align*}
\mathcal{R}_{3}\left(\Y\X^{-1},\X^{-1}\Y,\Y\right)=&\Rr_3(\Y)\Rr_1\left(\X^{-1}\right)^2+\Rr_3(\Y)
\Rr_2\left(\X^{-1},\X^{-1}\right)\\+&\Rr_1(\Y)\Rr_2\left(\X^{-1},\X^{-1}\right)\Rr_2(\Y)+
\Rr_2(\Y)\Rr_2\left(\X^{-1},\X^{-1}\right)\R_1(\Y)=\\
&\frac{2\lambda^2}{(\lambda-1)^5}+\frac{2\lambda}{(\lambda-1)^5}+\frac{\lambda}{(\lambda-1)^4}+
\frac{\lambda}{(\lambda-1)^4}=\frac{4\lambda^2}{(\lambda-1)^5}.
\end{align*}
Which gives
\begin{align*}
\mathcal{R}_{3} \left(\I+\Y\X^{-1},\I+\X^{-1}\Y,\X+\Y\right)=\frac{4\lambda^2}{(\lambda-1)^5}.
\end{align*}
One can also check that
\begin{align*}
&\mathcal{R}_{2}(\I+\X^{-1}\Y,\X+\Y)\mathcal{R}_{1}(\I+\Y\X^{-1})=
\mathcal{R}_{2}(\I+\Y\X^{-1},\X+\Y)\mathcal{R}_{1}(\I+\X^{-1}\Y)=\\
&\left(\Rr_2\left(\Y\X^{-1},\X\right)+\Rr_2\left(\Y\X^{-1},\Y\right)\right)
\left(1+\Rr_1\left(\X^{-1}\right)\Rr_1(\Y)\right)=\\
&\left(-\Rr_1(\Y)\Rr_1(\X)+\Rr_1\left(\X^{1}\right)\Rr_2(\Y)\right)\left(1+\Rr_1\left(\X^{-1}\right)\Rr_1(\Y)\right)
=\frac{2\lambda-1}{(\lambda-1)^4}
\end{align*}
So the cumulant $\mathcal{R}_{2}\left((\I+\Y\X^{-1})(\I+\X^{-1}\Y),\X+\Y\right)$ is not equal to $0$.
\end{proof}
\subsection*{Acknowledgement} The author thanks J. Weso\l{}owski for many helpful comments and discussions. This research was partially supported by NCN
grant 2012/05/B/ST1/00554.
\bibliographystyle{plain}
\bibliography{Bibl}
\end{document}